\documentclass[11pt]{amsart}

\usepackage{amsmath,amssymb,amsthm}
\usepackage{microtype}
\usepackage[hidelinks]{hyperref}
\hypersetup{
  pdftitle={Finite-index extensions of essentially free countable Borel equivalence relations},
  pdfauthor={Jason Zesheng Chen}
}

\newtheorem{theorem}{Theorem}
\newcommand{\leB}{\leq_B}
\newcommand{\Fin}{\operatorname{Fin}}

\title[Finite-index extensions and essential freeness]{Finite-index extensions of essentially free\\
countable Borel equivalence relations}
\author{Jason Zesheng Chen}
\date{September 2026}

\begin{document}

\begin{abstract}
We show that finite-index extensions of essentially free countable
Borel equivalence relations are essentially free, answering a question
of Kechris. 
\end{abstract}

\maketitle

A countable Borel equivalence relation is \emph{free} if it is induced
by a free Borel action of a countable group. It is \emph{essentially
free} if it is Borel bireducible with a free countable Borel equivalence
relation, or equivalently if it is Borel reducible to one; see
\cite[Chapter~10]{Kechris}. For countable Borel equivalence relations
$E\subseteq F$ on the same standard Borel space, we say that $E$ has
\emph{finite index} in $F$ if each $F$-class contains only finitely many
$E$-classes; no uniform bound is assumed. Kechris asks whether essential
freeness is preserved by such extensions \cite[Problem~10.4]{Kechris}.

\begin{theorem}\label{thm:main}
Let $E\subseteq F$ be countable Borel equivalence relations. If $E$
is essentially free and has finite index in $F$, then $F$ is
essentially free.
\end{theorem}

\begin{proof}
For a standard Borel space $X$, let $\Fin(X)$ be the standard Borel
space of its nonempty finite subsets. Given any countable Borel
equivalence relation $R$ on $X$, define $R^{\mathrm{fin}}$ on
$\Fin(X)$ by
$$
 a\mathrel{R^{\mathrm{fin}}}b
 \quad\Longleftrightarrow\quad
 \{[z]_R:z\in a\}=\{[z]_R:z\in b\}.
$$

Now let $E, F$ be as given. By a result of Kaya \cite[Proposition~11]{Kaya}, we have $F\leB E^{\mathrm{fin}}$.

Since $E$ is essentially free, there exists a free countable Borel equivalence relation $Q$ on some standard Borel space $Y$, and a Borel
reduction $f$ of $E$ to $Q$. The map $a\mapsto f[a]$ reduces
$E^{\mathrm{fin}}$ to $Q^{\mathrm{fin}}$, as also noted by Kaya.
It therefore suffices to show that $Q^{\mathrm{fin}}$ is essentially
free.

Let $\Gamma\curvearrowright Y$ be a free Borel action inducing $Q$.
For each $n\geq 1$, put
$$
 D_n=\{(y_i)_{i<n}\in Y^n:
      i\neq j\Longrightarrow\neg(y_i\mathrel Q y_j)\}.
$$
The wreath product $\Gamma^n\rtimes\mathfrak S_n$ acts on $D_n$ by
$$
 ((\gamma_i)_{i<n},\sigma)\cdot(y_i)_{i<n}
   =(\gamma_i\cdot y_{\sigma^{-1}(i)})_{i<n}.
$$
This action is free: if an element fixes a tuple, then
$y_i\mathrel Q y_{\sigma^{-1}(i)}$ for every $i$, so the pairwise
inequivalence of the coordinates forces $\sigma$ to be the identity.
Freeness of the original action then gives $\gamma_i=1_\Gamma$ for
all $i$. Let $R_n$ be the resulting orbit equivalence relation.
Two tuples are $R_n$-equivalent exactly when their coordinates
represent the same collection of $Q$-classes.

Fix a Borel linear order on $Y$. For each $a\in\Fin(Y)$, choose the
least element of $a$ in each $Q$-class that it meets, and list the
chosen elements in increasing order. This is a Borel map into the
disjoint union of the spaces $D_n$, and the preceding description
of $R_n$ shows that it is a reduction:
$$
 Q^{\mathrm{fin}}\leB\bigoplus_{n\geq 1}R_n.
$$
Essential freeness is closed under countable disjoint sums and
Borel reducibility \cite[Chapter~10]{Kechris}. Since each $R_n$ is
free, $Q^{\mathrm{fin}}$, and hence $F$, is essentially free.
\end{proof}

Note that the finiteness assumption cannot be replaced by countable index.
Indeed, every countable Borel equivalence relation $F$ on $X$ extends
the equality relation $\Delta_X$ with countable index, and $\Delta_X$
is free, being induced by the trivial group. 

We conclude by remarking that the present argument does not apply to the corresponding question for treeability is posed in
\cite[Section~6.4(B)]{JKL}. The point used above is that the relevant
wreath-product actions remain free, but the freeness of the wreath-product actions gives no treeing of their orbit equivalence relations.


\begin{thebibliography}{9}

\bibitem{JKL}
S.~Jackson, A.~S. Kechris, and A.~Louveau,
\emph{Countable Borel equivalence relations},
J. Math. Log. \textbf{2} (2002), no.~1, 1--80.

\bibitem{Kaya}
B.~Kaya,
\emph{The complexity of the topological conjugacy problem for
Toeplitz subshifts},
Israel J. Math. \textbf{220} (2017), no.~2, 873--897.

\bibitem{Kechris}
A.~S. Kechris,
\emph{The Theory of Countable Borel Equivalence Relations},
Cambridge Tracts in Mathematics, vol.~234,
Cambridge University Press, Cambridge, 2024.

\end{thebibliography}
\end{document}